\documentclass[10pt,oneside,a4paper]{amsart}
\usepackage{latexsym}
\usepackage{amsfonts,amsmath,amssymb,indentfirst}
\usepackage[english]{babel}
\usepackage[all]{xy}
\usepackage[pdftex]{hyperref}
\newcommand{\bdism}{\begin{displaymath}}
\newcommand{\edism}{\end{displaymath}}

\newcommand{\rr}{\mathbb{R}}
\newcommand{\qq}{\mathbb{Q}}
\newcommand{\zz}{\mathbb{Z}}
\newcommand{\pp}{\mathbb{P}}
\newcommand{\oo}{\mathcal{O}}
\newcommand{\jj}{\mathcal{J}}

\DeclareMathOperator{\nklt}{Nklt}
\DeclareMathOperator{\supp}{Supp}
\DeclareMathOperator{\lcm}{lcm}
\DeclareMathOperator{\bs}{Bs}
\DeclareMathOperator{\B}{\mathbf{B}}
\newtheorem{theorem}{Theorem}[section]
\newtheorem{proposition}[theorem]{Proposition}
\newtheorem{corollary}[theorem]{Corollary}

\newtheorem{remark}[theorem]{Remark}
\newtheorem{definition}[theorem]{Definition}
\newtheorem{question}[theorem]{Question}
\newtheorem{conjecture}[theorem]{Conjecture}
\author{\scshape Gabriele Di Cerbo}
\title{\bf Uniform bounds for the Iitaka fibration}
\begin{document}
\pagestyle{headings}
\begin{abstract}
We give effective bounds for the uniformity of the Iitaka fibration. These bounds follow from an effective theorem on the birationality of some adjoint linear series. In particular we derive an effective version of the main theorem in \cite{Pac}.
\end{abstract}
\date{November 28, 2011}
\maketitle

\section{Introduction}
\pagenumbering{arabic}
Hacon and M$^{\text{c}}$Kernan in \cite{Hac2} proposed the following conjecture.
\begin{conjecture}\label{1}
There is a positive integer $m_{n,\kappa}$ such that for any $m\geq m_{n,\kappa}$ sufficiently divible, $\phi_{|mK_{X}|}$ is birationally equivalent to the Iitaka fibration of $X$ for all smooth projective varieties $X$ of dimension $n$ and Kodaira dimension $\kappa$.
\end{conjecture}
They proved in \cite{Hac} the case when $\kappa(X)=n$. For different proofs see also \cite{Tak} and \cite{Tsu}. Their result is not effective and it remains a difficult question to find bounds for these numbers. When $\kappa(X)<n$ the standard approach to this problem is to use the canonical bundle formula of Fujino and Mori \cite{Fuj}. Roughly speaking, it says that for the Iitaka fibration $f:X\dashrightarrow Y$ the question is equivalent to finding a uniform bound for the birationality of linear systems on $Y$ of type $K_{Y}+M_{Y}+B_{Y}$ where $(Y,B_{Y})$ is a klt pair and $M_{Y}$ is a nef $\qq$-divisor. With this approach $m_{n,\kappa}$ depends also on some other numerical parameters which appear in the formula, see Section 4 for details. Under these new assumptions there are some partial results. Fujino and Mori proved the case $\kappa(X)=1$. Some years later Viehweg and Zhang in \cite{Vie} proved the case $\kappa(X)=2$. If $\dim(X)=3$ Rigler gave a different proof in \cite{Rin}. For low dimensional varieties, i.e. $\dim(X)\leq 4$, the log version of Conjecture \ref{1} has been studied in \cite{Tod} and \cite{Tod1}. The first result for arbitrary Kodaira dimension is due to Pacienza in \cite{Pac} but he needs to assume that $K_{Y}$ is pseudo-effective and $M_{Y}$ is big. Recently Jiang in \cite{Jia} proved the case where $M_{Y}$ is numerically trivial by reducing the problem to a result on log pluricanonical maps in \cite{Hac1}. In summary, the canonical bundle formula suggests that we need some general theorems for adjoint linear systems in order to prove Conjecture \ref{1}. In fact, thanks to the following result, Pacienza derived his theorem on the uniformity of the Iitaka fibration.
\begin{theorem}[Pacienza \cite{Pac}]\label{2}
For any positive integers $n$ and $\nu$, there exists an integer $m_{n,\nu}$ such that for any smooth complex projective variety $X$ of dimension $n$ with pseudo-effective canonical divisor, and any big and nef $\qq$-divisor $M$ on $X$ such that $\nu M$ is a $\zz$-divisor, the pluriadjoint map
\bdism
\phi_{m(K_{X}+M)}:X\dashrightarrow \pp H^{0}(X,\oo_{X}(m(K_{X}+M)))
\edism
is birational for all $m\geq m_{n,\nu}$ divisible by $\nu$.
\end{theorem}
His proof relies on some techniques developed by Takayama in \cite{Tak} and Debarre \cite{Deb}. There are some deep results involved, like Takayama's extension theorem and the weak positivity theorem of Campana \cite{Cam}. Unfortunately all these theorems allow us only to derive non-effective statements. On the other hand Koll\'ar's proof in \cite{Kol} of Angehrn-Siu's theorem is effective but it only deals with big and nef divisors. In this note we explain how to use the method of Koll\'ar to derive an effective version of Pacienza's theorem. Furthermore our proof relies on more elementary techniques. Our main result is:
\begin{theorem}\label{3}
Let $(X,\Delta)$ be a klt pair. Let $M$ be a big and nef $\zz$-divisor on $X$ and $E$ a pseudo-effective $\qq$-divisor on $X$. Then for any
\bdism
m>\binom{n+2}{2}
\edism
the map induced by $|\left\lceil K_{X}+\Delta+E+mM\right\rceil|$ is birational.
\end{theorem}
Taking $X$ smooth, $\Delta=0$ and $E=(m-1)K_{X}$ we get an effective version of Theorem \ref{2}. \\
We now show how Theorem \ref{3} gives a uniform result for the Iitaka fibration. Let $f:X\dashrightarrow Y$ be the Iitaka fibration of $X$. As we will see in Section 4 there are two positive integers $b$ and $N$, depending only on the general fiber of $f$, such that $bNM_{Y}$ is an integral divisor. For the definition of $b$ and $N$ see Definition \ref{def}. Furthermore we will see in Proposition \ref{12} that $|mK_{X}|$ gives a map birationally equivalent to the Iitaka fibration if and only if the linear series $|\left\lfloor m(K_{Y}+B_{Y}+M_{Y})\right\rfloor|$ gives a birational map. Then Theorem \ref{3} implies the following.
\begin{theorem}\label{4}
Let $f:X\dashrightarrow Y$ be the Iitaka fibration of $X$, where $X$ is a smooth projective variety of Kodaira dimension $\kappa$. Suppose $K_{Y}+B_{Y}$ is pseudo-effective and $M_{Y}$ is big. Then for any
\bdism
m>bN\binom{\kappa+2}{2}
\edism
divisible by $bN$, the pluricanonical map $\phi_{|mK_{X}|}$ is birationally equivalent to the Iitaka fibration.
\end{theorem}
One can give conditions only on $X$ and the generic fiber of $f$ such that Theorem \ref{4} applies, see for example Corollary \ref{15}. Of course we would like to prove a similar statement without the assumption $K_{Y}+B_{Y}$ pseudo-effective. In Section 3 we study the pseudo-effective threshold of $(X,\Delta)$ with respect to a big and nef divisor $M$. In particular we obtain a similar result if we assume that the pseudo-effective threshold is bounded away from one.

\section{Pluriadjoint maps}
We follow the notation and terminology of \cite{Kol} and \cite{KM}. However we state here some definitions we will need later.
\begin{definition}
A pair $(X,\Delta)$ consists of a normal variety $X$ and a $\qq$-Weil divisor $\Delta\geq 0$ such that $K_{X}+\Delta$ is $\qq$-Cartier.
\end{definition}
The multiplier ideal of a divisor $D$ on a normal variety $X$ is denoted by $\jj(X,D)$. We refer to \cite{Laz2} for the definition. 
\begin{definition}
The non-klt locus $\nklt(X,\Delta)$ of a pair $(X,\Delta)$ is
\bdism
\nklt(X,\Delta):=\left\{x\in X\:|\:\text{$(X,\Delta)$ in not klt at $x$}\right\}.
\edism
\end{definition}
We will use the following relation
\bdism
\nklt(X,\Delta)=\supp(\oo_{X}/\jj(X,\Delta))_{\text{red}}.
\edism
See \cite{Laz2} Section 9.3.B for the proof.
\begin{proposition}\label{5}
Let $(X,\Delta)$ be a pair. Let $M$ be a big and nef Cartier divisor on $X$ and $N$ be a Cartier divisor on $X$ such that $N-K_{X}-\Delta$ is pseudo-effective. Let $x_{1}$ and $x_{2}$ be two general points in $X$. Suppose there are $t_{0}>0$ and an effective $\qq$-divisor $D_{0}$ such that 
\begin{enumerate}
\item $D_{0}\sim_{\qq} t_{0}M$;
\item $x_{1},x_{2}\in \nklt(X,\Delta+D_{0})$;
\item $x_{1}$ is an isolated point in $\nklt(X,\Delta+D_{0})$.
\end{enumerate}
Then for any $m>t_{0}$ the linear system $|N+mM|$ separates $x_{1}$ and $x_{2}$.
\end{proposition}
\begin{proof}
Let $E$ a pseudo-effective $\qq$-divisor such that $N=K_{X}+\Delta+E$. Fix $m>t_{0}$ and write $D:=D_{0}+E$. Note that $D$ is equivalent to an effective $\qq$-divisor. Let $V:=\nklt(X,\Delta+D)$. Let $x_{1}$ and $x_{2}$ be two general points not contained in $\supp(E)$, then we have that $x_{1},x_{2}\in V$ and $x_{1}$ is isolated in $V$. In order to get separation of points we want the following map to be surjective
\bdism
H^{0}(X,\oo_{X}(N+mM))\rightarrow H^{0}(V,\oo_{X}(N+mM)|_{V}).
\edism
It fits in the long exact sequence given by
\bdism
0\rightarrow \oo_{X}(N+mM)\otimes \jj(X,\Delta+D)\rightarrow \oo_{X}(N+mM)\rightarrow \oo_{X}(N+mM)|_{V} \rightarrow 0,
\edism
then it is enough to prove that 
\bdism
H^{1}(X,\oo_{X}(N+mM)\otimes \jj(X,\Delta+D))=0.
\edism
Since 
\bdism
N+mM-(K_{X}+\Delta+D)\sim_{\qq}(m-t_{0})M
\edism
is big and nef, the above vanishing follows from Nadel vanishing on singular varieties, see Theorem 2.16 in \cite{Kol} or Theorem 9.4.17 in \cite{Laz2}.
\end{proof}
\begin{remark}
In Proposition \ref{5} we work with $D_{0}\sim_{\qq} t_{0}M$ instead of working with $D_{0}\sim_{\qq} t_{0}(M+E)$ as Pacienza does in his Lemma 6.3. This is a crucial difference between our approach and that of Pacienza.
\end{remark}
We recall a result of Koll\'ar in \cite{Kol}, Theorem 6.5. 
\begin{theorem}[Koll\'ar]\label{6}
Let $(X,\Delta)$ be a projective klt pair and $M$ a big and nef $\qq$-Cartier $\qq$-divisor on $X$. Let $x_{1}$ and $x_{2}$ be closed points in $X$ and $c(k)$ positive numbers such that if $Z\subset X$ is an irreducible subvariety with $x_{1}\in Z$ or $x_{2}\in Z$ then
\bdism
(M^{\dim Z}\cdot Z)>c(\dim Z)^{\dim Z}.
\edism
Assume also that 
\bdism
\sum^{n}_{k=1}\sqrt[k]{2}\frac{k}{c(k)}\leq 1.
\edism
Then there is an effective $\qq$-divisor $D\sim_{\qq} M$ such that:
\begin{enumerate}
\item $x_{1},x_{2}\in \nklt(X,\Delta+D)$;
\item $x_{1}$ is an isolated point in $\nklt(X,\Delta+D)$.
\end{enumerate}
\end{theorem}
We recall the definition of the augmented base locus $\B_{+}(M)$, see Definiton 10.3.2 in \cite{Laz2}.
\begin{definition}
The stable base locus of a divisor $M$ is 
\bdism
\B(M):=\bigcap_{m\geq 1}\bs(|mM|),
\edism
where $\bs(|M|)$ is the base locus of $M$. \\
The augmented base locus of a divisor $M$ is the Zariski-closed set 
\bdism
\B_{+}(M):=\B(M-\epsilon A),
\edism
for any ample $A$ and sufficiently small $\epsilon>0$.
\end{definition}
Theorem \ref{6} easily implies the following useful result.
\begin{corollary}\label{7}
Let $(X,\Delta)$ be a klt pair. Let $M$ be a big and nef Cartier divisor on $X$. Then for any $x_{1},x_{2}\notin \B_{+}(M)$ there exists an effective $\qq$-divisor $D_{0}$ with 
\begin{enumerate}
\item $D_{0}\sim_{\qq} \binom{n+2}{2} M$;
\item $x_{1},x_{2}\in \nklt(X,D_{0}+\Delta)$;
\item $x_{1}$ is an isolated point in $\nklt(X,D_{0}+\Delta)$.
\end{enumerate}
\end{corollary}
\begin{proof}
Recall that $\B_{+}(M)$ is a proper subset of $X$ if and only if $M$ is big, see Example 1.7 in \cite{Ein}. Furthermore
\bdism
\B_{+}(M)=\bigcap_{M=A+E}\supp(E),
\edism
where the intersection is taken over all decomposition $M=A+E$, where $A$ is ample and $E$ effective, see Remark 1.3 in \cite{Ein}. Then for any variety $Z$ through $x_{1},x_{2}\notin \B_{+}(M)$ we have that $M^{\dim(Z)}\cdot Z>0$. Since $M$ is integral then $(M^{\dim Z}\cdot Z)\geq 1$. Using the following inequality
\bdism
\sum^{n}_{k=1}\sqrt[k]{2}k<\sum_{k=1}^{n}\left(1+\frac{1}{k}\right)k=\binom{n+2}{2}-1,
\edism
we see that the divisor $\binom{n+2}{2} M$ satifies the conditions of Theorem \ref{6} with $c(k)=\binom{n+2}{2}$. 
\end{proof}
Now our main theorem is a consequence of the above results.
\begin{proof}[Proof of Theorem \ref{3}]
By Corollary 1.4.3 in \cite{Bchm} we can assume that $(X,\Delta)$ is a $\qq$-factorial klt pair, see also Corollary 4.4 in \cite{Loh} for a more detailed proof. \\
We can write
\bdism
\left\lceil K_{X}+\Delta+E+mM\right\rceil=K_{X}+\Delta+E'+mM,
\edism
where 
\bdism
E':=E+\left\lceil K_{X}+\Delta+E\right\rceil-(K_{X}+\Delta+E)
\edism
is a pseudo-effective $\qq$-divisor. Then Proposition \ref{5} and Corollary \ref{7} imply that $|\left\lceil K_{X}+\Delta+E+mM\right\rceil|$ separates any two points $x_{1}$ and $x_{2}$ not in $\B_{+}(M)$. Since $X-\B_{+}(M)$ is a dense open subset of $X$, the result follows.
\end{proof}
\begin{corollary}\label{8}
Let $(X,\Delta)$ be a klt pair with $K_{X}+\Delta$ pseudo-effective. Let $M$ be a big and nef $\qq$-divisor on $X$.  Let $\nu$ be an integer such that $\nu M$ is a $\zz$-divisor, then for any
\bdism
m>\nu\binom{n+2}{2}
\edism
divisible by $\nu$, the map induced by $|\left\lceil m(K_{X}+\Delta+M)\right\rceil|$ is a birational map.
\end{corollary}
\begin{proof}
Let $m$ be as in the statement and set $E:=(m-1)(K_{X}+\Delta)$. Then Theorem \ref{3} gives the result.
\end{proof}
Note that if we take $X$ smooth and $\Delta=0$, we get an effective version of Theorem \ref{2}. \\ 
Of course one can ask a weaker question about the non-vanishing of the cohomology group $H^{0}(X,\oo_{X}(\left\lceil m(K_{X}+\Delta+M)\right\rceil))$. Using a similar version of Proposition \ref{5} and Theorem 6.4 in \cite{Kol} we get:
\begin{theorem}\label{not}
Let $(X,\Delta)$ be a klt pair with $K_{X}+\Delta$ pseudo-effective. Let $M$ be a big and nef $\qq$-divisor on $X$. Let $\nu$ be an integer such that $\nu M$ is a $\zz$-divisor, then for any
\bdism
m>\nu\binom{n+1}{2}
\edism
divisible by $\nu$, $|\left\lceil m(K_{X}+\Delta+M)\right\rceil|\neq\emptyset$.
\end{theorem}
In the application to the Iitaka fibration we need to study the round down of these linear series instead of the round up. 
\begin{definition}
The index of a variety $X$ is the smallest natural number $a(X)$ such that $a(X)K_{X}$ is a Cartier divisor.
\end{definition}
\begin{corollary}\label{9}
Let $(X,\Delta)$ be a klt pair such that $K_{X}+\Delta$ is pseudo-effective. Let $M$ be a big and nef $\qq$-divisor on $X$ and let $\nu$ be an integer such that $\nu M$ is a $\zz$-divisor. Suppose $\left\lfloor k\Delta\right\rfloor\geq(k-1)\Delta$ for any $k\in\zz_{>0}$ divisible by $\nu$ and $a(X)$. Then for any
\bdism
m>\nu \binom{n+2}{2}
\edism
divisible by $\nu$ and $a(X)$ the map induced by $|\left\lfloor  m(K_{X}+\Delta+M)\right\rfloor|$ is birational.
\end{corollary}
\begin{proof}
Let $m$ be as in the statement, then we can write
\bdism
\left\lfloor m(K_{X}+\Delta+M)\right\rfloor=K_{X}+(m-1)(K_{X}+\Delta)+\left\lfloor m\Delta\right\rfloor-(m-1)\Delta+mM. 
\edism
Let $E:=(m-1)(K_{X}+\Delta)+\left\lfloor m\Delta\right\rfloor-(m-1)\Delta$ and note that $K_{X}+E+mM$ is an $\zz$-divisor. Then the result follows from Theorem \ref{3}.
\end{proof}

\section{Pseudo-effective Threshold}
In this section we deal with the case where $K_{X}+\Delta$ is not pseudo-effective. Following \cite{Bchm} and \cite{Vie} we define the pseudo-effective threshold. 
\begin{definition}
Let $(X,\Delta)$ be a pair such that $K_{X}+\Delta$ is not pseudo-effective. Let $M$ be a big divisor on $X$. We define the pseudo-effective threshold $e(X,\Delta,M)$ of $(X,\Delta)$ with respect to $M$ as 
\bdism
e(X,\Delta,M):=\inf\left\{e'\in \rr \:|\: K_{X}+\Delta+e' M \:\text{is pseudo-effective}\right\}.
\edism
\end{definition}
If there is no risk of confusion we denote it only by $e(M)$.
\begin{proposition}\label{10}
Let $(X,\Delta)$ be a klt pair such that $K_{X}+\Delta$ is not pseudo-effective. Let $M$ be a big and nef $\zz$-divisor on $X$ such that $K_{X}+\Delta+M$ is big. Let $e(M)$ be the pseudo-effective threshold of $(X,\Delta)$ with respect to $M$. Then for any 
\bdism
m>\frac{1}{1-e(M)}\binom{n+2}{2}
\edism
the map induced by the linear system $\left|\left\lceil m(K_{X}+\Delta+M)\right\rceil\right|$ is birational.
\end{proposition}
\begin{proof}
Let $m$ be as in the statement. Since $K_{X}+\Delta+M$ is big, by Corollary 2.2.24 in \cite{Laz1}, we know that $e(M)<1$. Then we can find a rational number $e'$ such that $e(M)\leq e'<1$ and 
\bdism
m>\frac{1}{1-e'}\binom{n+2}{2}.
\edism
We can write
\bdism
m(K_{X}+\Delta+M)=K_{X}+\Delta+(m-1)(K_{X}+\Delta+e'M)+(m(1-e')+e')M.
\edism
In particular it is enough to prove that the map induced by round up of the linear system 
\bdism
K_{X}+\Delta+(m-1)(K_{X}+\Delta+e'M)+m(1-e')M
\edism
is a birational map. Then the result follows from Theorem \ref{3}.
\end{proof}
We have an analogue statement for the round down.
\begin{proposition}\label{11}
Let $(X,\Delta)$, $M$ and $e(M)$ as in Proposition 3.2. Furthermore assume that $\left\lfloor k\Delta\right\rfloor\geq(k-1)\Delta$ for any $k\in\zz_{>0}$ divisible by $a(X)$. Then for any 
\bdism
m>\frac{1}{1-e(M)}\binom{n+2}{2}
\edism
divisible by $a(X)$, the map induced by the linear system $\left|\left\lfloor m(K_{X}+\Delta+M)\right\rfloor\right|$ is birational.
\end{proposition}
\begin{proof}
The argument is the same as in Proposition \ref{10} and Corollary \ref{9}.
\end{proof}

\section{Iitaka Fibration}
We now show how the previous results gives the uniformity of the Iitaka fibration under some extra conditions. For the definition and basic properties of the Iitaka fibration we refer to \cite{Laz1}. We recall the canonical bundle formula and some of his properties, see \cite{Fuj} for details. Let $f:X\dashrightarrow Y$ be the Iitaka fibration of $X$ with $Y$ nonsingular and general fiber $F$. Blowing up $X$ we may assume that $f$ is a morphism. Then the canonical bundle formula says that there are $\qq$-divisors $B_{Y}$ and $M_{Y}$ such that
\bdism
K_{X}\sim_{\qq}f^{*}(K_{Y}+B_{Y}+M_{Y}).
\edism
$B_{Y}$ is called the boundary divisor and it is an effective divisor such that $(Y,B_{Y})$ is a klt pair. $M_{Y}$ is called the moduli part and it is a nef $\qq$-divisor. We now define two numbers which play a key role in the canonical bundle formula. 
\begin{definition}\label{def}
Let 
\bdism
b:=\min\left\{b'>0\:|\:\left|b'K_{F}\right|\neq\emptyset\right\}.
\edism
Let $B$ be the $(n-\kappa(X))$-th Betti number of a non-singular model of the cover $E\rightarrow F$ associated to the unique element of $|bK_{F}|$. We define
\bdism
N=N(B):=\lcm\left\{m\in\zz_{>0}\:|\:\varphi(M)\leq B\right\},
\edism
where $\varphi$ is the Euler function.
\end{definition}
We list some properties we will need later.
\begin{proposition}\label{12} 
The following hold true:
\begin{enumerate}
\item $bN M_{Y}$ is a $\zz$-divisor;
\item for any $m\in\zz_{>0}$ divisible by $b$, we have 
\bdism
H^{0}(X,\oo_{X}(mK_{X}))\cong H^{0}(Y,\oo_{Y}(\left\lfloor m(K_{Y}+M_{Y}+B_{Y})\right\rfloor));
\edism
\item for any $m\in\zz_{>0}$ divisible by $bN$, $\left\lfloor mB_{Y}\right\rfloor-(m-1)B_{Y}$ is effective;
\item $K_{Y}+M_{Y}+B_{Y}$ is a big $\qq$-divisor;
\item if $F$ has a good minimal model and $Var(f)$ is maximal then $M_{Y}$ is big.
\end{enumerate}
\end{proposition} 
\begin{proof}
For $(1)$, $(2)$ and $(3)$ see \cite{Fuj}. $(4)$ follows from $(2)$. Finally $(5)$ follows from a theorem of Kawamata in \cite{Kaw}. See also Corollary 3.1 in \cite{Pac}.
\end{proof}
In particular $(2)$ implies that $|mK_{X}|$ is birational to the Iitaka fibration if and only if $|\left\lfloor m(K_{Y}+M_{Y}+B_{Y})\right\rfloor|$ gives a birational map.
\begin{proof}[Proof of Theorem \ref{4}]
It is just Proposition \ref{12} and Corollary \ref{9} with $a(Y)=1$ because $Y$ is smooth.
\end{proof}
We can now prove an effective version of Theorem 1.2 in \cite{Pac}.
\begin{corollary}\label{15}
Let $X$ be a smooth projective variety of Kodaira dimension $\kappa$ and $f:X\dashrightarrow Y$ be the Iitaka fibration. Assume that
\begin{enumerate}
\item $Y$ is not uniruled;
\item f has maximal variation;
\item the generic fiber $F$ has a good minimal model.
\end{enumerate}
Then for any
\bdism
m>bN\binom{\kappa+2}{2}
\edism
divisible by $bN$, the pluricanonical map $\phi_{|mK_{X}|}$ is birationally equivalent to $f$.
\end{corollary}
\begin{proof}
The main result in \cite{Dem} implies that if $Y$ is not uniruled then $K_{Y}$ is pseudo-effective. By Proposition \ref{12} we know that $M_{Y}$ is big. Then Theorem \ref{4} applies.
\end{proof}
If we use Theorem \ref{not} instead of Theorem \ref{3} we can prove the following.
\begin{theorem}
Let $X$ as in Theorem \ref{4}. Then for any
\bdism
m>bN\binom{\kappa+1}{2}
\edism
divisible by $bN$, the cohomology group $H^{0}(X,\oo_{X}(mK_{X}))$ is non-zero.
\end{theorem}
In particular if $\kappa (X)=n-2$ we can choose $b=12$ and $N=22$.\\
If $K_{Y}+B_{Y}$ is not pseudo-effective we have a similar result but the bound depends on the pseudo-effective threshold $e(M_{Y})$.
\begin{theorem}\label{13}
Let $f:X\dashrightarrow Y$ the Iitaka fibration of $X$ with general fiber $F$. Assume that 
\begin{enumerate}
\item f has maximal variation;
\item the generic fiber $F$ has a good minimal model.
\end{enumerate}
Then for any sufficiently divisible
\bdism
m>\frac{bN}{1-e(M_{Y})}\binom{\kappa+2}{2}
\edism
the map associated to $|mK_{X}|$ is birational to the Iitaka fibration.
\end{theorem}
\begin{proof}
Apply Proposition \ref{12} and Proposition \ref{11} with $a(Y)=1$.
\end{proof}
It is now natural to ask the following.
\begin{question}\label{14}
Is possible to find a universal bound $e<1$ which depends only on the dimension of $Y$, $b$ and $N$ such that $e(Y,B_{Y},M_{Y})\leq e$ for any $Y$, $B_{Y}$ and $M_{Y}$ as in Theorem \ref{13}?
\end{question}
Viehweg and Zhang gave an affirmative answer to Question \ref{14} in the case $\dim(Y)=2$, see Lemma 2.10 and Lemma 2.11 in \cite{Vie}.
\section*{Acknowledgments}
First I would like to express my gratitude to Professor J\'anos Koll\'ar for his constant support and many enlightening discussions. I also would like to thank Professor Gianluca Pacienza for constructive comments on the paper.

\noindent
{Princeton University, Princeton NJ 08544-1000}

\noindent{gdi@math.princeton.edu}

\end{document}